\documentclass[10pt]{amsart}
\usepackage[T1]{fontenc}
\usepackage[latin5]{inputenc}
\usepackage{float}
\usepackage{amssymb}
\usepackage{amsfonts}
\usepackage[centertags]{amsmath}
\usepackage{amsthm}
\usepackage{newlfont}
 \makeatletter
\theoremstyle{plain}
\newtheorem{thm}{Theorem}[section]
\numberwithin{equation}{section}
\numberwithin{figure}{section}  
\theoremstyle{plain}

\theoremstyle{plain}
\newtheorem{cor}[thm]{Corollary} 
\theoremstyle{plain}
\newtheorem{defin}[thm]{Definition}
\theoremstyle{plain}
\newtheorem{lem}[thm]{Lemma} 
\theoremstyle{plain}

\begin{document}
\title{On the Generalized Difference Matrix Domain on Strongly Almost Convergent Double Sequence Spaces}
\author{Orhan Tu\v{g}}
\subjclass[2010]{46A45, 40C05}
\keywords{Four-dimensional matrices; $B(r,s,t,u)$-Summable; strongly almost convergent double sequences.}
\address[Orhan Tu\v{g}]{Department of Mathematics Education, Tishk International University, Qazi Mohammad, Erbil, Iraq}
\email{orhan.tug@ishik.edu.iq, tugorhan27@gmail.com}
\begin{abstract}
Most recently, some new double sequence spaces $B(\mathcal{M}_{u})$, $B(\mathcal{C}_{\vartheta})$ where $\vartheta=\{b,bp,r,f,f_0\}$ and $B(\mathcal{L}_{q})$ for $0<q<\infty$  have been introduced as four-dimensional generalized difference matrix  $B(r,s,t,u)$ domain on the double sequence spaces $\mathcal{M}_{u}$, $\mathcal{C}_{\vartheta}$ where $\vartheta=\{b,bp,r,f,f_0\}$ and $\mathcal{L}_{q}$ for $0<q<\infty$, and some topological properties, dual spaces, some new four-dimensional matrix classes and matrix transformations related to these spaces have also been studied by Tu\u{g} and Ba\c{s}ar and Tu\u{g} (see \cite{OT,OT2,Orhan,Orhan 2}). In this present paper, we introduce new strongly almost null and strongly almost convergent double sequence spaces  $B[\mathcal{C}_f]$ and $B[\mathcal{C}_{f_0}]$ as domain of four-dimensional generalized difference matrix $B(r,s,t,u)$ in the spaces $[\mathcal{C}_f]$ and $[\mathcal{C}_{f_0}]$, respectively. Firstly, we prove that the new double sequence spaces $B[\mathcal{C}_f]$ and $B[\mathcal{C}_{f_0}]$ are Banach spaces with its norm. Then, we give some inclusion relations including newly defined strongly almost convergent double sequence spaces. Moreover, we calculate the $\alpha-$dual, $\beta(bp)-$dual and $\gamma-$dual of the space $B[\mathcal{C}_f]$. Finally, we characterize new four-dimensional matrix classes $([\mathcal{C}_{f}];\mathcal{C}_{f})$, $([\mathcal{C}_{f}];\mathcal{M}_{u})$, $(B[\mathcal{C}_{f}];\mathcal{C}_{f})$, $(B[\mathcal{C}_{f}];\mathcal{M}_{u})$ and we complete this work with some significant results.
\end{abstract}\maketitle

\section{Preliminaries, Background and Notation}
By the set $\Omega:=\{x=(x_{mn}): x_{mn}\in\mathbb{C},~~\forall m,n\in\mathbb{N}\}$, we denote all complex valued double sequences. $\Omega$ is a vector space with coordinatewise addition and scalar multiplication and any vector subspace of $\Omega$ is called a double sequence space. A double sequence $x=(x_{mn})$ is called convergent in the Pringsheim's sense to a limit point $L$, if for every $\epsilon>0$ there exists a natural number $n_0=n_0(\epsilon)$ and $L\in\mathbb{C}$ such that $|x_{mn}-L|<\epsilon$ for all $m,n>n_0$,	where $\mathbb{C}$ denotes the complex field. The space of all convergent double sequence in Pringsheim's sense is denoted by $\mathcal{C}_p$, that is,
	\begin{eqnarray*}
		\mathcal{C}_p:=\{x=(x_{mn})\in\Omega: \exists L\in\mathbb{C},\forall\epsilon>0 \exists k\in\mathbb{N}, \forall m,n\geq k \ni |x_{mn}-L|<\epsilon\}
	\end{eqnarray*}
which is a linear space with coordinatewise addition and scalar multiplication. Moricz\cite{MF} proved that the double sequence space $\mathcal{C}_p$ is complete seminormed space with the seminorm
	\begin{eqnarray*}
		\|x\|_{\infty}=\lim_{N\to\infty}\sup_{m,n\geq N}|x_{mn}|.
	\end{eqnarray*}

We must note here that the space of all null double sequences in Pringsheim's sense is denoted by $\mathcal{C}_{p0}$. 

A double sequence $x=(x_{mn})$ of complex number is called bounded if $\|x\|_{\infty}=\sup_{m,n\in\mathbb{N}}|x_{mn}|<\infty$, where $\mathbb{N}=\{0,1,2,\cdots\}$ and the space of all bounded double sequences is denoted by $\mathcal{M}_u$, that is;
	\begin{eqnarray*}
		\mathcal{M}_u:=\{x=(x_{mn})\in\Omega: \|x\|_{\infty}=\sup_{m,n\in\mathbb{N}}|x_{m,n}|<\infty\}
	\end{eqnarray*}
	which is a Banach space with the $\|x\|_{\infty}$ norm.

Unlike single sequence there are such double sequences which are convergent in Pringsheim's sense but unbounded. That is, the set $\mathcal{C}_p\setminus\mathcal{M}_u$ is not empty. Boos \cite{BJ} defined the sequence $x=(x_{mn})$ by
\begin{eqnarray*}
	x_{mn}=\left\{
	\begin{array}{ccl}
		n&, & m=0, n\in\mathbb{N}; \\
		0&, & m\geq1, n\in\mathbb{N}.
	\end{array}
	\right.
\end{eqnarray*}
which is obviously in $\mathcal{C}_p$, i.e., $p-\lim_{m,n\rightarrow\infty}x_{mn}=0$ but not in the set $\mathcal{M}_u$, i.e., $\|x\|_{\infty}=\sup_{m,n\in\mathbb{N}}|x_{mn}|=\infty$. Thus,  $x\in\mathcal{C}_p-\mathcal{M}_u $.

Let consider the set $\mathcal{C}_{bp}$ of double sequences which are both convergent in Pringsheim's sense and bounded, i.e., $\mathcal{C}_{bp}=\mathcal{C}_{p}\cap\mathcal{M}_{u}$ which is defined as
	\begin{eqnarray*}
		\mathcal{C}_{bp}:=\{x=x_{mn}\in\mathcal{C}_{p}: \|x\|_{\infty}=\sup_{m,n\in\mathbb{N}}|x_{mn}|<\infty\}=\mathcal{C}_{p}\cap\mathcal{M}_{u}.
	\end{eqnarray*}
The set of all convergent in Pringsheim's sense and bounded double sequence space $\mathcal{C}_{bp}$ is a linear Banach space with the norm
	\begin{eqnarray*}
		\|x\|_{\infty}=\sup_{m,n\in\mathbb{N}}|x_{mn}|<\infty.
	\end{eqnarray*}

Hardy \cite{HGH} introduced a sequence in the space $\mathcal{C}_p$ which is called regularly convergent if it is a single convergent sequence with respect to each index. We denote the set of such double sequences by $\mathcal{C}_r$,i.e.,
	\begin{eqnarray*}
		\mathcal{C}_r:=\{x=x_{mn}\in\mathcal{C}_{p}\: \forall m\in\mathbb{N}\ni(x_{mn})_{m}\in c,\textit{ and } \forall n\in\mathbb{N}\ni(x_{mn})_{n}\in c\}.
	\end{eqnarray*}
Regular convergence requires the boundedness of double sequence that is the main difference between regular convergence and the convergence in Pringsheim' sense. We can also note here that  $\mathcal{C}_{bp0}=\mathcal{M}_{u}\cap\mathcal{C}_{p0}$ and $\mathcal{C}_{r0}=\mathcal{C}_{r}\cap\mathcal{C}_{p0}$.

The space $\mathcal{L}_q$ of all absolutely $q-$summable double sequences was introduced by Ba\c{s}ar and Sever \cite{FBYS} as follow;
	\begin{eqnarray*}
		\mathcal{L}_q:=\left\{x=(x_{kl})\in\Omega:\sum_{k,l}|x_{kl}|^q<\infty\right\},\quad (1\leq q<\infty)
	\end{eqnarray*}
which is a Banach space with the norm $\|\cdot\|_q$ defined by
	\begin{equation*}
	\|x\|_q=\left(\sum_{k,l}|x_{kl}|^q\right)^{1/q}.
	\end{equation*}
Moreover, Zeltser \cite{ZM} introduced the space $\mathcal{L}_u$ which derived from the space $\mathcal{L}_q$ with $q=1$.

The double sequence spaces $\mathcal{BS}$, $\mathcal{CS}_{\vartheta}$; where $\vartheta=\{p, bp, r\}$, and $\mathcal{BV}$ were introduced by Altay and Ba\c{s}ar \cite{Altay}. The set $\mathcal{BS}$ of all bounded series whose sequences of partial sums are bounded is defined by
	\begin{equation*}
	\mathcal{BS}=\left\{x=(x_{kl})\in\Omega:\sup_{m,n\in\mathbb{N}}|s_{mn}|<\infty\right\}
	\end{equation*}
where the sequence $s_{mn}=\sum_{k,l=0}^{m,n}x_{kl}$ is the $(m,n)-th$ partial sum of the series. The series space $\mathcal{BS}$ is a linear Banach space with norm defined as
\begin{equation}\label{eq2.9999}
\|x\|_{\mathcal{BS}}=\sup_{m,n\in\mathbb{N}}\left|\sum_{k,l=0}^{m,n}x_{kl}\right|,
\end{equation}
which is linearly isomorphic to the sequence space $\mathcal{M}_u$. The set $\mathcal{CS_{\vartheta}}$ of all series whose sequences of partial sums are $\vartheta-$convergent in Pringsheim's sense is defined by
	\begin{equation*}
	\mathcal{CS_{\vartheta}}=\left\{x=(x_{kl})\in\Omega: (s_{mn})\in\mathcal{C_{\vartheta}}\right\}
	\end{equation*}
where $\vartheta=\{p, bp, r\}$. The space $\mathcal{CS}_p$ is linear complete seminormed space with the seminorm defined by
\begin{equation*}
\|x\|_\infty=\lim_{n\to\infty}\left(\sup_{k,l\geq n}\left|\sum_{i,j=0}^{k,l}x_{ij}\right|\right),
\end{equation*}
which is isomorphic to the  sequence space $\mathcal{C}_p$. Moreover, the sets $\mathcal{CS}_{bp}$ and $\mathcal{CS}_{r}$ are also linear Banach spaces with the norm (\ref{eq2.9999}) and the inclusion $\mathcal{CS}_{r}\subset \mathcal{CS}_{bp}$ holds. The set $\mathcal{BV}$ of all double sequences of bounded variation is defined by
	\begin{equation*}
	\mathcal{BV}=\left\{x=(x_{kl})\in\Omega: \sum_{k,l}\left|x_{kl}-x_{k-1,l}-x_{k,l-1}+x_{k-1,l-1}\right|<\infty\right\}.
	\end{equation*}
The space $\mathcal{BV}$ is linear Banach space with the norm defined by
\begin{equation*}
\|x\|_{\mathcal{BV}}=\sum_{k,l}\left|x_{kl}-x_{k-1,l}-x_{k,l-1}+x_{k-1,l-1}\right|,
\end{equation*}
which is linearly isomorphic to the space $\mathcal{L}_u$ of absolutely convergent double series. Moreover, the inclusions $\mathcal{BV}\subset \mathcal{C_{\vartheta}}$ and $\mathcal{BV}\subset \mathcal{M}_u$ are strictly hold.

For any double sequence spaces $\lambda$ and $\mu$, the following set $D_2(\lambda,\mu)$ defines as follows
\begin{eqnarray}\label{eq1.10}
D_2(\lambda,\mu)=\{ a=(a_{mn})\in\Omega :ax=(a_{mn}x_mn)\in\mu\textit{ for all }x=(x_{mn})\in\lambda \}
\end{eqnarray}
is called multiplier space of the double sequence spaces $\lambda$ and $\mu$. One can be observed for a double sequence space $u$ with $\mu\subset u\subset\lambda$ that the inclusions
\begin{eqnarray*}
	D_2(\lambda,\mu)\subset D_2(u,\mu)~and~D_2(\lambda,\mu)\subset D_2(\lambda, u)
\end{eqnarray*}
hold. By means of the set (\ref{eq1.10}) $\alpha-$dual, $\beta(\vartheta)-$dual with respect to the $\vartheta-$convergence and $\gamma-$dual of a double sequence space $\lambda$ which are denoted by $\lambda^{\alpha}$, $\lambda^{\beta(\vartheta)}$ and $\lambda^{\gamma}$, respectively, are defined as
\begin{eqnarray*}
	\lambda^{\alpha}:=D_2(\lambda,\mathcal{L}_u),~\lambda^{\beta(\vartheta)}:=D_2(\lambda,\mathcal{CS}_{\vartheta}),~~and~~\lambda^{\gamma}:=D_2(\lambda,\mathcal{BS})
\end{eqnarray*}

Moreover, let $\lambda$ and $\mu$ are arbitrary double sequences with $\lambda\supset\mu$ such that $\lambda^{\alpha}\subset\mu^{\alpha}$, $\lambda^{\gamma}\subset\lambda^{\alpha}$ and $\lambda^{\beta(\vartheta)}\subset\lambda^{\alpha}$ holds. But the inclusion $\lambda^{\gamma}\subset\lambda^{\beta(\vartheta)}$ does not hold, since the $\vartheta-$convergence of the double sequence does not guarantee its boundedness.

In here, we concern the four dimensional matrix transformation from any double sequence space $\lambda$ to any double sequence space $\mu$. Given any four-dimensional infinite matrix $A=(a_{mnkl})$, where $m,n,k,l\in\mathbb{N}$, any double sequence $x=(x_{kl})$, we write $Ax=\{(Ax)_{mn}\}_{m,n\in\mathbb{N}}$, the $A-$transform of $x$, exists for every sequence $x=(x_{kl})\in\lambda$ and it is in $\mu$; where
\begin{eqnarray}\label{eq1.1}
(Ax)_{mn}=\vartheta-\sum_{k,l}a_{mnkl}x_{kl}~\textit{ for each } m,n\in\mathbb{N}.
\end{eqnarray}
The four dimensional matrix domain has fundamental importance for this article. Therefore, this concept is presented in this paragraph.The $\vartheta-$summability domain $\lambda_A^{(\vartheta)}$ of $A$ in a space $\lambda$ of double sequences is described as
\begin{eqnarray*}
	\lambda_A^{(\vartheta)}=\left\{x=(x_{kl})\in\Omega: Ax=\left(\vartheta-\sum_{k,l}a_{mnkl}x_{kl}\right)_{m,n\in\mathbb{N}}\textit{ exists and is in }\lambda\right\}.
\end{eqnarray*}
The notation (\ref{eq1.1}) says that $A$ maps the space $\lambda$ into the space $\mu$ if $\lambda\subset\mu_A^{(\vartheta)}$ and we denote the set of all four-dimensional matrices, transforming the space $\lambda$ into the space $\mu$, by $(\lambda:\mu)$. Thus, $A=(a_{mnkl})\in (\lambda:\mu)$ if and only if the double series on the right side of (\ref{eq1.1}) converges in the sense of $\vartheta$ for each $m,n\in\mathbb{N}$, i.e, $A_{mn}\in\lambda^{\beta(\vartheta)}$ for all $m,n\in\mathbb{N}$ and we have $Ax\in\mu$ for all $x\in\lambda$; where $A_{mn}=(a_{mnkl})_{k,l\in\mathbb{N}}$ for all $m,n\in\mathbb{N}$. Moreover, the following definitions are significant in order to classify the four dimensional matrices. A four-dimensional matrix $A$ is called $\mathcal{C}_{\vartheta}-conservative$ if $\mathcal{C}_{\vartheta}\subset (\mathcal{C}_{\vartheta})_A$, and is called $\mathcal{C}_{\vartheta}-regular$ if it is $\mathcal{C}_{\vartheta}-conservative$ and
\begin{eqnarray*}
	\vartheta-\lim Ax=\vartheta-\lim_{m,n\to\infty}(Ax)_{mn}=\vartheta-\lim_{m,n\to\infty}x_{mn}, \textit{ where } x=(x_{mn})\in\mathcal{C}_{\vartheta}.
\end{eqnarray*}

The aim of calculating the matrix domain on sequence spaces is to set up new sequence spaces which can be expansion or contraction of the original space. There are several four-dimensional matrices such as Riesz mean, Euler mean, etc., to calculate its domain on double sequence spaces. The most common and used one is triangular matrices which have been defined by Adams \cite{ACR}. An infinite matrix $A=(a_{mnkl})$ is called a triangular matrix if $a_{mnkl}=0$ for $k>m$ or $l>n$ or both. We also say by \cite{ACR} that an infinite matrix $A=(a_{mnkl})$ is said to be a triangular if $a_{mnmn}\neq0$ for all $m,n\in\mathbb{N}$. Moreover, Cooke \cite{CRC} proved that every infinite triangular matrix has a unique right and left inverse which are equal triangular matrices.

Concerning the matrix domain on sequence spaces, we should say here that such studies on single sequence spaces are much more than the studies on double sequence spaces. There are a few  works on four-dimensional matrix domains on double sequence spaces which have been done by several mathematician. To review the concerning literature about the domain $\lambda_A$ of four-dimensional infinite matrix $A$ in double sequence spaces $\lambda$, the following table gives the summary of the works conducted so far.
\begin{center}
	\begin{tabular}{|c|c|c|c|}
		\hline $A$ &$\lambda$ &$\lambda_A$ & \textrm{refer to:} \\\hline
		$C$ & $\mathcal{M}_u$, $\mathcal{C}_p$, $\mathcal{C}_{0p}$, $\mathcal{C}_r$, $\mathcal{C}_{bp}$, $\mathcal{L}_{q}$ &$\tilde{\mathcal{M}_u}$, $\tilde{\mathcal{C}_p}$, $\tilde{\mathcal{C}_{0p}}$, $ \tilde{\mathcal{C}_r}$, $\tilde{\mathcal{C}_{bp}}$, $\tilde{\mathcal{L}_{q}}$& \cite{Mursaleen}\\

		$\Delta(1,-1,1,-1)$ & $\mathcal{M}_u$, $\mathcal{C}_p$, $\mathcal{C}_{0p}$, $\mathcal{C}_r$, $\mathcal{L}_{q}$&$\mathcal{M}_u(\Delta)$, $\mathcal{C}_p(\Delta)$, $\mathcal{C}_{0p}(\Delta)$, $ \mathcal{C}_r(\Delta)$, $\mathcal{L}_{q}(\Delta)$& \cite{Demiriz}\\

		$C$ & $\tilde{\mathcal{M}_u}$, $\tilde{\mathcal{C}_p}$, $\tilde{\mathcal{C}_{0p}}$, $ \tilde{\mathcal{C}_r}$, $\tilde{\mathcal{C}_{bp}}$, $\tilde{\mathcal{L}_{q}}$&$\tilde{\mathcal{M}_u(t)}$, $\tilde{\mathcal{C}_p(t)}$, $\tilde{\mathcal{C}_{0p}(t)}$, $ \tilde{\mathcal{C}_r(t)}$, $\tilde{\mathcal{C}_{bp}(t)}$, $\tilde{\mathcal{L}_{q}(t)}$& \cite{Demiriz 2}\\

		$R^{qt}$& $\mathcal{L}_s$ & $R^{qt}(\mathcal{L}_s))$ & \cite{MYFB22}\\
		
		$B(r,s,t,u)$ & $\mathcal{M}_u$, $\mathcal{C}_p$, $\mathcal{C}_{bp}$, $\mathcal{C}_r$, $\mathcal{L}_q$&$B(\mathcal{M}_u)$, $B(\mathcal{C}_p)$, $B(\mathcal{C}_{bp})$, $B(\mathcal{C}_r)$, $B(\mathcal{L}_q)$& \cite{OT2}\\

		$R^{qt}$ & $\mathcal{M}_u$, $\mathcal{C}_p$, $\mathcal{C}_{bp}$, $\mathcal{C}_r$ &$(\mathcal{M}_u)_{R^{qt}}$, $(\mathcal{C}_p)_{R^{qt}}$, $(\mathcal{C}_{bp})_{R^{qt}}$, $(\mathcal{C}_r)_{R^{qt}}$& \cite{Medine}\\
		
		$E(r,s)$ & $\mathcal{L}_p$, $\mathcal{M}_u$ & $\mathcal{E}_p^{r,s}$, $\mathcal{E}_{\infty}^{r,s}$& \cite{Galebi}\\

		$B(r,s,t,u)$ & $\mathcal{C}_f$, $\mathcal{C}_{f_0}$&$B(\mathcal{C}_f)$, $B(\mathcal{C}_{f_0})$& \cite{Orhan, Orhan 2}\\
		
		$\Delta(1,-1,1,-1)$ & $\mathcal{F}$, $\mathcal{F}_{0}$, $[\mathcal{F}]$, $[\mathcal{F}_{0}]$ &$\mathcal{F}(\Delta)$, $\mathcal{F}_{0}(\Delta)$, $[\mathcal{F}](\Delta)$, $[\mathcal{F}_{0}](\Delta)$& \cite{Husamettin}\\
		
		\hline
	\end{tabular}
	\\~
	
\end{center}
where the matrices $C$, $\Delta(1,-1,1,-1)$, $R^{qt}$  and $E(r,s)$ denote C\'{e}saro mean, Four-Dimensional difference matrix, Riesz mean and doubly Euler mean, respectively.

The four dimensional generalized difference matrix $B(r,s,t,u)=\{b_{mnkl}(r,s,t,u)\}$, as a generalization of $\Delta(1,-1,1,-1)$, introduced by Tu\v{g} and Ba\c{s}ar\cite{OT}. Then, Tu\v{g}\cite{OT2,Orhan,Orhan 2,OT3} calculated the $B(r,s,t,u)-$domain on some double sequence spaces and stated some significant topological properties, inclusion relations, dual spaces and matrix transformations beside characterizing some new four-dimensional matrix classes.

The matrix $B(r,s,t,u)=\{b_{mnkl}(r,s,t,u)\}$ is represented by
\begin{eqnarray*}
	b_{mnkl}(r,s,t,u):=\left\{
	\begin{array}{ccl}
		su&, & (k,l)=(m-1,n-1), \\
		st&, & (k,l)=(m-1,n), \\
		ru&, & (k,l)=(m,n-1), \\
		rt&, & (k,l)=(m,n) \\
		0&, &\textrm{otherwise}
	\end{array}\right.
\end{eqnarray*}
for $r,s,t,u\in\mathbb{R}\backslash\{0\}$ and for all $m,n,k,l\in\mathbb{N}$. The matrix  $B(r,s,t,u)$ transforms a double sequence $x=(x_{mn})$ as
\begin{eqnarray}\label{eq1.3}
y_{mn}:=\{B(r,s,t,u)x\}_{mn}&=&\sum_{k,l}b_{mnkl}(r,s,t,u)x_{kl}\\
&=&sux_{m-1,n-1}+stx_{m-1,n}+rux_{m,n-1}+rtx_{mn}\notag
\end{eqnarray}
for all $m,n\in\mathbb{N}$. To obtain the relation the relation between $x=(x_{mn})$ and $y=(y_{mn})$, it is needed to calculate the inverse of $B(r,s,t,u)$ which is the matrix $F=f_{mnkl}(r,s,t,u)$ such that
\begin{eqnarray*}\label{eq1.2}
	f_{mnkl}(r,s,t,u):=\left\{
	\begin{array}{ccl}
		\frac{(-s/r)^{m-k}(-u/t)^{n-l}}{rt}&, & 0\leq k\leq m,~ 0\leq l\leq n, \\
		0&, & \textrm{otherwise}
	\end{array}\right.
\end{eqnarray*}
for all $m,n,k,l\in\mathbb{N}$. Thus, the relation between $x=(x_{mn})$ and $y=(y_{mn})$ can be obtained by 
\begin{eqnarray}\label{eq1.4}
x_{mn}=\frac{1}{rt}\sum_{k,l=0}^{m,n}\left(\frac{-s}{r}\right)^{m-k}\left(\frac{-u}{t}\right)^{n-l}y_{kl}~\textrm{ for all}~m,n\in\mathbb{N}.
\end{eqnarray}

Note that the four-dimensional generalized difference matrix $B(r,s,t,u)$ will be four-dimensional difference matrix $\Delta(1,-1,1,-1)$ in the case $r=t=1, ~s=u=-1$. Therefore, the results obtained by the matrix $B(r,s,t,u)$ is much more general then the results abtained by the matrix $\Delta(1,-1,1,-1)$. Throughout the paper, the connection between the double sequence $x=(x_{mn})$ and $y=(y_{mn})$ will be given by the relation (\ref{eq1.4}).

\section{The Sequence Spaces of Almost and Strongly Almost Convergent Double Sequences}
The concept of almost convergence for single sequence introduced by Lorentz\cite{GGL} and then Rhoades\cite{MFRBE} extended the idea of almost convergence for double sequence. He stated that a double sequence $x=(x_{kl})$ of complex numbers is called almost convergent to a generalized limit $L$ if
\begin{eqnarray*}
	p-\lim_{q,q'\to\infty}\sup_{m,n>0}\left|\frac{1}{(q+1)(q'+1)}\sum_{k=m}^{m+q}\sum_{l=n}^{n+q'}x_{kl}-L\right|=0.
\end{eqnarray*}
In this case, $L$ is called the $f_2-$limit of the double sequence $x$. Then Ba\c{s}arir \cite{Basarir} defined the concept of strongly almost convergence of double sequences. A double sequence $x=(x_{kl})$ of real numbers is said to be strongly almost convergent to a limit $L_1$ if 
\begin{eqnarray*}
	p-\lim_{q,q'\to\infty}\sup_{m,n>0}\frac{1}{(q+1)(q'+1)}\sum_{k=m}^{m+q}\sum_{l=n}^{n+q'}\left|x_{kl}-L\right|=0.
\end{eqnarray*}
and it is uniform in $m,n\in\mathbb{N}$. Now we may define the set of all almost convergent, almost null, strongly almost convergent and strongly almost null double sequences, respectively, as follow;

\begin{eqnarray*}
&&\mathcal{C}_f:=\left\{\begin{array}{c}
x=(x_{kl})\in\Omega:\exists L\in\mathbb{C} \ni p-\lim_{q,q'\to\infty}\sup_{m,n>0}\left|\frac{1}{(q+1)(q'+1)}\sum_{k=m}^{m+q}\sum_{l=n}^{n+q'}x_{kl}-L\right|=0,\\
	~\textit{uniformly in }m,n\in\mathbb{N}\textit{  for some }L
		\end{array}\right\},\\
&&\mathcal{C}_{f_0}:=\left\{\begin{array}{c}
		x=(x_{kl})\in\Omega:\exists L\in\mathbb{C} \ni p-\lim_{q,q'\to\infty}\sup_{m,n>0}\left|\frac{1}{(q+1)(q'+1)}\sum_{k=m}^{m+q}\sum_{l=n}^{n+q'}x_{kl}\right|=0,\\
		~\textit{uniformly in }m,n\in\mathbb{N}
	\end{array}\right\},\\
&&[\mathcal{C}_f]:=\left\{\begin{array}{c}
		x=(x_{kl})\in\Omega:\exists L\in\mathbb{C} \ni p-\lim_{q,q'\to\infty}\sup_{m,n>0}\frac{1}{(q+1)(q'+1)}\sum_{k=m}^{m+q}\sum_{l=n}^{n+q'}\left|x_{kl}-L_1\right|=0,\\
		~\textit{uniformly in }m,n\in\mathbb{N}\textit{  for some }L_1
	\end{array}\right\},\\
&&[\mathcal{C}_{f_0}]:=\left\{\begin{array}{c}
		x=(x_{kl})\in\Omega:\exists L\in\mathbb{C} \ni p-\lim_{q,q'\to\infty}\sup_{m,n>0}\frac{1}{(q+1)(q'+1)}\sum_{k=m}^{m+q}\sum_{l=n}^{n+q'}\left|x_{kl}\right|=0,\\
		~\textit{uniformly in }m,n\in\mathbb{N}
	\end{array}\right\},
\end{eqnarray*}

Here we can say for this case that $L_1$ is called $[f_2]-$limit of a double sequence $x=(x_{kl})$ and written shortly as $[f_2]-\lim x=L_1$. 

Here we state some geometrical and topological properties of these sets. It is well known that $\mathcal{C}_{p}\setminus\mathcal{C}_f$ is not empty, but the inclusions $\mathcal{C}_{bp}\subset\mathcal{C}_f\subset \mathcal{M}_u$ strictly hold. Since the following inequality 
\begin{eqnarray*}
\sup_{m,n>0}\left|\frac{1}{(q+1)(q'+1)}\sum_{k=m}^{m+q}\sum_{l=n}^{n+q'}x_{kl}-L\right|\leq \sup_{m,n>0}\frac{1}{(q+1)(q'+1)}\sum_{k=m}^{m+q}\sum_{l=n}^{n+q'}\left|x_{kl}-L\right|.
\end{eqnarray*}
holds, we can easily say that if a double sequence is strongly almost convergent, that is, the right hand side of the above inequality approaches to zero if we pass to limit as $q,q'\to\infty$, then the left hand side of the inequality also tends to zero. It says that the inclusion $[\mathcal{C}_f]\subset\mathcal{C}_f$ holds and it easily can be seen that the double sequence $x_{kl}=(-1)^l$, for all $k\in\mathbb{N}$, is in $\mathcal{C}_f\setminus[\mathcal{C}_f]$. So the inclusion is strictly hold. Now, we can mention here that the inclusions $\mathcal{C}_{bp}\subset[\mathcal{C}_{f_0}]\subset[\mathcal{C}_f]\subset\mathcal{C}_{f_0}\subset\mathcal{C}_f\subset \mathcal{M}_u$ are strictly hold and each inclusion is proper. 

Furthermore, the sets $\mathcal{C}_f$ and $\mathcal{C}_{f_0}$ are Banach spaces with the norm
\begin{eqnarray*}
\|x\|_{\mathcal{C}_f}=\sup_{q,q',m,n\in\mathbb{N}}\left|\frac{1}{(q+1)(q'+1)}\sum_{k=m}^{m+q}\sum_{l=n}^{n+q'}x_{kl}\right|.
\end{eqnarray*}
and the sets $[\mathcal{C}_f]$ and $[\mathcal{C}_{f_0}]$ are Banach spaces with norm
\begin{eqnarray*}
	\|x\|_{[\mathcal{C}_f]}=\sup_{q,q',m,n\in\mathbb{N}}\frac{1}{(q+1)(q'+1)}\sum_{k=m}^{m+q}\sum_{l=n}^{n+q'}\left|x_{kl}\right|.
\end{eqnarray*}

\v{C}unjalo \cite{FC} introduced that a double sequence $x=(x_{kl})$ is called almost Cauchy if for every $\epsilon>0$ there exists a positive integer $K$ such that
\begin{eqnarray*}
	\left|\frac{1}{(q_1+1)(q_1'+1)}\sum_{k=m_1}^{m_1+q_1}\sum_{l=n_1}^{n_1+q_1'}x_{kl}-\frac{1}{(q_2+1)(q_2'+1)}\sum_{k=m_2}^{m_2+q_2}\sum_{l=n_2}^{n_2+q_2'}x_{kl}\right|<\epsilon
\end{eqnarray*}
for all $q_1, q_1', q_2, q_2'>K$ and $(m_1,n_1),(m_2,n_2)\in\mathbb{N}\times\mathbb{N}$. Then, Mursaleen and Mohiuddine \cite{MMSA} proved that every double sequence is almost convergent if and only if it is almost Cauchy.

M\'{o}ricz and Rhoades \cite{MFRBE} characterized the four-dimensional matrix class $(\mathcal{C}_f;\mathcal{C}_{bp})$ with $bp-\lim Ax=f_2-\lim x$. Then, Zeltser et al. \cite{ZMM} characterized the matrix classes $(\mathcal{C}_{\vartheta};\mathcal{C}_{f})$ which is called $\mathcal{C}_{\vartheta}-$conservative. If $f_2-\lim Ax=\vartheta-\lim x$ for all $x\in\mathcal{C}_{\vartheta}$, then it is called  $\mathcal{C}_{\vartheta}-$regular. Moreover, Mursaleen \cite{MM} introduced the almost strongly regularity for double sequences and characterized the matrix class $(\mathcal{C}_{f};\mathcal{C}_{f})$

\section{strongly almost B-summable double sequence spaces}

Almost $B-$summable double sequence spaces $B(\mathcal{C}_f)$ and $B(\mathcal{C}_{f_0})$ were defined and studied by Tu\v{g} \cite{Orhan} which was derived by the domain of four-dimensional generalized difference matrix $B(r,s,t,u)$ in the spaces of all almost convergent and almost null double sequences $\mathcal{C}_{f}$ and $\mathcal{C}_{f_0}$, respectively. 

In this paper, we define strongly almost $B-$summable double sequence spaces $B[\mathcal{C}_f]$ and $B[\mathcal{C}_{f_0}]$ as domain of four-dimensional generalized difference matrix $B(r,s,t,u)$ in the spaces of strongly almost convergent and strongly almost null double sequences $[\mathcal{C}_{f}]$ and $[\mathcal{C}_{f_0}]$, respectively.

Now we may define the new spaces $B[\mathcal{C}_f]$ and $B[\mathcal{C}_{f_0}]$ as follow:
\begin{eqnarray*}
&&B[\mathcal{C}_f]:=\left\{\begin{array}{c}
		x=(x_{kl})\in\Omega:\exists L\in\mathbb{C} \ni p-\lim_{q,q'\to\infty}\sup_{m,n>0}\frac{1}{(q+1)(q'+1)}\sum_{k=m}^{m+q}\sum_{l=n}^{n+q'}\left|(Bx)_{kl}-L\right|=0,\\
		~\textit{uniformly in }m,n\in\mathbb{N}\textit{  for some }L_1
	\end{array}\right\},\\
&&B[\mathcal{C}_{f_0}]:=\left\{\begin{array}{c}
		x=(x_{kl})\in\Omega:\exists L\in\mathbb{C} \ni p-\lim_{q,q'\to\infty}\sup_{m,n>0}\frac{1}{(q+1)(q'+1)}\sum_{k=m}^{m+q}\sum_{l=n}^{n+q'}\left|(Bx)_{kl}\right|=0,\\
		~\textit{uniformly in }m,n\in\mathbb{N}
	\end{array}\right\},
\end{eqnarray*}

Now we may state the following essential theorem without proof. The proof can be done with quite similar way as in the \cite[Theorem 3.1, p.6]{Orhan}. 

\begin{thm}\label{thm2.0}
	The sequence spaces $B[\mathcal{C}_f]$ and $B[\mathcal{C}_{f_0}]$ are Banach spaces and linearly norm isomorphic to the spaces $[\mathcal{C}_{f}]$ and $[\mathcal{C}_{f_0}]$, respectively, with the norm defined by
	\begin{eqnarray}\label{eq2.1}
	\|x\|_{B[\mathcal{C}_f]}=\sup_{q,q',m,n\in\mathbb{N}}\frac{1}{(q+1)(q'+1)}\sum_{k=m}^{m+q}\sum_{l=n}^{n+q'}\left|(Bx)_{kl}\right|.
	\end{eqnarray}
\end{thm}

\begin{thm}
Let $s=-r, t=-u$. The inclusions $\mathcal{M}_{u}\subset B(\mathcal{C}_{f_0})$ strictly holds.
\end{thm}

\begin{proof}
First, we should show that the inclusion $\mathcal{M}_{u}\subset B(\mathcal{C}_{f_0})$ holds if $s=-r, t=-u$. Let $x=(x_{kl})\in\mathcal{M}_u$, that is, there exists a positive real number $M$ such that $\|x\|_{\infty}=\sup_{k,l\in\mathbb{N}}|x_{kl}|\leq M<\infty$. Now we should show that $x=(x_{kl})\in B(\mathcal{C}_{f_0})$ which says $(Bx)_{kl}\in\mathcal{C}_{f_0}$. Then, one can derive from the following inequality that
\begin{eqnarray*}
p&-&\lim_{q,q'\to\infty}\sup_{m,n>0}\left|\frac{1}{(q+1)(q'+1)}\sum_{k=m}^{m+q}\sum_{l=n}^{n+q'}(Bx)_{kl} \right|\\
&\leq& p-\lim_{q,q'\to\infty}\sup_{m,n>0}\frac{rt}{(q+1)(q'+1)}\left|x_{m-1,n-1}-x_{m+q,n-1}-x_{m-1,n+q'}+x_{m+q,n+q'} \right|\\
&\leq&p-\lim_{q,q'\to\infty}\sup_{m,n>0}\frac{4rtM}{(q+1)(q'+1)}=0
\end{eqnarray*}
clearly $x=(x_{kl})\in B(\mathcal{C}_{f_0})$.

Now we should state here that the inclusion is strict, that is, there is a sequence $x=(x_{kl})\in B(\mathcal{C}_{f_0})\setminus \mathcal{M}_u$. Let consider the sequence $x_{kl}=\frac{k}{rt}$, for all $l\in\mathbb{N}$ which is clearly not in $\mathcal{M}_u$. We can obtain from the $B-$transform of $x$ that
\begin{eqnarray*}
	(Bx)_{k,l}=\{B(r,-r,t,-t)x\}_{kl}&=&rtx_{k-1,l-1}-rtx_{k-1,l}-rtx_{k,l-1}+rtx_{kl}\\
	&=&rt\frac{(k-1)}{rt}-rt\frac{(k-1)}{rt}-rt\frac{k}{rt}+rt\frac{k}{rt}=0.
\end{eqnarray*}
Clearly we have the consequence that $p-\lim_{q,q'\to\infty}\sup_{m,n>0}\left|\sum_{k=m}^{m+q}\sum_{l=n}^{n+q'}(Bx)_{kl}/(q+1)(q'+1) \right|=0$. This completes the proof.
\end{proof}

\begin{thm}
Let $s=-r, t=-u$. The spaces $\mathcal{M}_u$ and $\mu$ do not contain each other where $\mu=\{B[C_{f}], B[C_{f_0}]\}$. 
\end{thm}

\begin{proof}
To proof this theorem, it is essential to show that there is at least one sequence in $B[\mathcal{C}_{f_0}]\cap \mathcal{M}_u$, $B[\mathcal{C}_{f_0}]\setminus \mathcal{M}_u$, and $\mathcal{M}_u\setminus B[\mathcal{C}_{f_0}]$. If we consider the sequences $e$ and $\frac{k(-1)^l}{rt}$, then clearly these sequences belongs to the sets $B[\mathcal{C}_{f_0}]\cap \mathcal{M}_u$, $B[\mathcal{C}_{f_0}]\setminus \mathcal{M}_u$, respectively. Now, let consider the following sequence $x=(x_{kl})\in\mathcal{M}_u$ as
\begin{eqnarray*}
x_{kl}=\left\{\begin{array}{ccl}
 1&,& ~~k+l~even,\\
0&,&~~othervise
\end{array}\right.
\end{eqnarray*}
Then, clearly we have
\begin{eqnarray*}
p-\lim_{q,q'\to\infty}\sup_{m,n>0}\frac{1}{(q+1)(q'+1)}\sum_{k=m}^{m+q}\sum_{l=n}^{n+q'}\left|(Bx)_{kl} \right|&=&p-\lim_{q,q'\to\infty}\sup_{m,n>0}\frac{rt}{(q+1)(q'+1)}\sum_{k=m}^{m+q}\sum_{l=n}^{n+q'}\left|2(-1)^{k+l} \right|\\
&=&2rt
\end{eqnarray*}
Therefore, $x\in\mathcal{M}_u\setminus B[\mathcal{C}_{f_0}]$.This completes the proof.
\end{proof}

\begin{thm}
Let $s=-r, t=-u$. The following inclusion relations strictly hold.
\begin{enumerate}
\item[(i)] $[\mathcal{C}_{f}]\subset B[\mathcal{C}_{f}]$ .
\item[(ii)]  $[\mathcal{C}_{f_0}]\subset B[\mathcal{C}_{f_0}]$.
\item[(iii)]  $B[\mathcal{C}_{f}]\subset B(\mathcal{C}_{f})$.
\item[(iv)] $B[\mathcal{C}_{f_0}]\subset B(\mathcal{C}_{f_0})$.
\end{enumerate}
\end{thm}

\begin{proof}
First, we consider to prove the inclusion $(i)$. Suppose that the sequence $x=(x_{kl})\in[\mathcal{C}_{f}]$. Then we have,
\begin{eqnarray*}
p-\lim_{q,q'\to\infty}\sup_{m,n>0}\frac{1}{(q+1)(q'+1)}\sum_{k=m}^{m+q}\sum_{l=n}^{n+q'}\left|x_{kl}-L\right|=0
\end{eqnarray*}
is uniform in $m,n\in\mathbb{N}$ for some $L\in\mathbb{C}$. Now, we need to show that  $B-$transform of $x=(x_{kl})$ is also in $[\mathcal{C}_{f}]$. Since $s=-r, t=-u$, we have
\begin{eqnarray*}
	p&-&\lim_{q,q'\to\infty}\sup_{m,n>0}\frac{1}{(q+1)(q'+1)}\sum_{k=m}^{m+q}\sum_{l=n}^{n+q'}\left|trx_{k-1,l-1}+-rtx_{k-1,l}-rtx_{k,l-1}+rtx_{kl}\right|\\
	&\leq&p-\lim_{q,q'\to\infty}\sup_{m,n>0}\frac{rt}{(q+1)(q'+1)}\sum_{k=m}^{m+q}\sum_{l=n}^{n+q'}\left|x_{k-1,l-1}-L\right|\\
	&+&p-\lim_{q,q'\to\infty}\sup_{m,n>0}\frac{rt}{(q+1)(q'+1)}\sum_{k=m}^{m+q}\sum_{l=n}^{n+q'}\left|x_{k-1,l}-L\right|\\
	&+&p-\lim_{q,q'\to\infty}\sup_{m,n>0}\frac{rt}{(q+1)(q'+1)}\sum_{k=m}^{m+q}\sum_{l=n}^{n+q'}\left|x_{k,l-1}-L\right|\\
	&+&p-\lim_{q,q'\to\infty}\sup_{m,n>0}\frac{rt}{(q+1)(q'+1)}\sum_{k=m}^{m+q}\sum_{l=n}^{n+q'}\left|x_{kl}-L\right|\\
	&=&0
\end{eqnarray*}
is uniform in $m,n\in\mathbb{N}$. It says that $x=(x_{kl})\in B[\mathcal{C}_{f}]$. The inclusion $(ii)$ can be shown in a similar way. Let define a double sequence 
\begin{eqnarray}\label{eq1.a}
x_{kl}=\frac{k(-1)^l}{rt}
\end{eqnarray}
It is clear that the double sequence $x_{kl}$ defined by (\ref{eq1.a}) is not bounded. But $B(r,s,t,u)-$transform of $x_{kl}$ is bounded if $s=-r, t=-u$. This consequence says that $x_{kl}$ is in $B[\mathcal{C}_{f_0}]\setminus[\mathcal{C}_{f_0}]$ and similarly in $B[\mathcal{C}_{f}]\setminus[\mathcal{C}_{f}]$.

The inclusions $(iii)$ and $(iv)$ is clearly seen by considering a double sequence $x=(x_{kl})\in B[\mathcal{C}_{f}](or~B[\mathcal{C}_{f_0}])$, says $(Bx)_{kl}\in[\mathcal{C}_{f}](or~[\mathcal{C}_{f_0}])$. Since $[\mathcal{C}_{f}](or
~[\mathcal{C}_{f_0}])\subset\mathcal{C}_{f}(or~\mathcal{C}_{f_0})$, then $(Bx)_{kl}\in\mathcal{C}_{f}(or~\mathcal{C}_{f_0})$, says $x_{kl}\in B(\mathcal{C}_{f})(or~B(\mathcal{C}_{f_0}))$ which says $B[\mathcal{C}_{f}]\subset B(\mathcal{C}_{f})$ and $B[\mathcal{C}_{f_0}]\subset B(\mathcal{C}_{f_0})$. Let define a double sequence $x_{kl}$ by
\begin{eqnarray}\label{eq1.b}
	(Bx)_{kl}=(-1)^l,\textit{ for all }k\in\mathbb{N},i.e.,(Bx)_{kl}=\begin{bmatrix}
	1 & -1 & 1 & -1 & 1 &\cdots\\
	1 & -1 & 1 & -1 & 1 &\cdots\\
	1 & -1 & 1 & -1 & 1 &\cdots\\
	1 & -1 & 1 & -1 & 1 &\cdots\\
	\vdots & \vdots & \vdots & \vdots & \vdots& \ddots\\
	\end{bmatrix},
\end{eqnarray}

Now, it is clear that the double sequence $(Bx)_{kl}$ defined by (\ref{eq1.b}) is in the set $\mathcal{C}_{f}\setminus[\mathcal{C}_{f}]$. Therefore, $x=(x_{kl})\in B(\mathcal{C}_{f})\setminus B[\mathcal{C}_{f}]$ where $x_{kl}=\frac{1}{rt}\sum_{i,j=0}^{k,l}\left(\frac{-s}{r}\right)^{k-i}\left(\frac{-u}{t}\right)^{l-j}(-1)^j$ for all $k\in\mathbb{N}$. This concludes the proof.
\end{proof}

\section{Dual spaces of the sequence space $B[\mathcal{C}_f]$}

In this present section, firstly, we calculate the $\alpha-$dual of the space $B[\mathcal{C}_f]$. Then, we state some needed Lemmas and notations to calculate the $\beta(bp)-$dual and $\gamma-$dual of the space $B[\mathcal{C}_f]$.

\begin{thm}
	Let $\left|s/r\right|,\left|u/t\right|<1$. The $\alpha-$dual of the space $B[\mathcal{C}_f]$ is the space $\mathcal{L}_u$.
\end{thm}

\begin{proof}
To prove $\left\{B[\mathcal{C}_f]\right\}^{\alpha}=\mathcal{L}_u$, we should show that the inclusions $\mathcal{L}_u\subset \left\{B[\mathcal{C}_f]\right\}^{\alpha}$ and $\left\{B[\mathcal{C}_f]\right\}^{\alpha}\subset \mathcal{L}_u$ hold. The first inclusion can be proved by the similar way as in \cite[Theorem 4.1]{Orhan}. So we pass the repetition.

For the second inclusion, suppose that $(a_{kl})\in \left\{B[\mathcal{C}_f]\right\}^{\alpha}\subset\mathcal{L}_u$. Then, we have $\sum_{k,l}|a_{kl}x_{kl}|<\infty$ for all $x=(x_{kl})\in B[\mathcal{C}_f]$. 

Let us define a double sequence $x=(x_{kl})$ as  $x_{kl}=\{k(-1)^{l}/(rt)\}$. for $\left|s/r\right|,\left|u/t\right|<1$, it is clear that $|(Bx)_{kl}|\leq 0$, says $x=(x_{kl})\in B[\mathcal{C}_f]$ but
	\begin{eqnarray*}
		\sum_{k,l}|a_{kl}x_{kl}|=\frac{1}{|rt|}\sum_{k,l}|a_{kl}|k=\infty.
	\end{eqnarray*}
This means that $(a_{kl})\notin \left\{B(\mathcal{C}_f)\right\}^{\alpha}$ which is a contradiction. Hence, the sequence $(a_{kl})$ must be in $\mathcal{L}_u$. So, the inclusion $\left\{B(\mathcal{C}_f)\right\}^{\alpha}\subset \mathcal{L}_u$ holds. This is what we proposed.
\end{proof}

\begin{lem}\cite[Theorem 2.2]{ZMM} The following statements hold:
	
(a) The matrix $A=(a_{mnkl})\in(\mathcal{C}_{bp};\mathcal{C}_{bp})$ if and only if it satifies the following conditions.
\begin{eqnarray}
\label{eq1.5}&&\sup_{m,n\in\mathbb{N}}\sum_{k,l}|a_{mnkl}|<\infty,\\
\label{eq1.6}&&bp-\lim_{m,n\to\infty}a_{mnkl}=a_{kl}~ exists~ for~ k,l\in\mathbb{N},\\
\label{eq1.7}&&bp-\lim_{m,n\to\infty}\sum_{k,l}a_{mnkl}=v~ exists,\\
\label{eq1.8}&&bp-\lim_{m,n\to\infty}\sum_{k}|a_{mn,k,l_0}-a_{k,l_0}|=0 ~for ~l_0\in\mathbb{N},\\
\label{eq1.9}&&bp-\lim_{m,n\to\infty}\sum_{l}|a_{mn,k_0,l}-a_{k_0,l}|=0~ for ~k_0\in\mathbb{N}
\end{eqnarray}
In this case, $a=(a_{kl})\in\mathcal{L}_u$ and 
\begin{eqnarray*}
bp-\lim_{m,n}[Ax]_{mn}=\sum_{k,l}a_{kl}x_{kl}+\left(v-\sum_{k,l}a_{kl} \right)bp-\lim_{mn}x_{mn}, ~(x\in\mathcal{C}_{bp})
\end{eqnarray*}

(b) The matrix  $A=(a_{mnkl})\in(\mathcal{C}_{bp};\mathcal{C}_{bp})$ and $bp-\lim Ax=bp-\lim_{mn}x_{mn},~~(x\in\mathbb{C}_{bp})$ if and only if the conditions (\ref{eq1.5})-(\ref{eq1.9}) hold with $a_{kl}=0$ for all $k,l\in\mathbb{N}$ and $v=1$.
\end{lem}

\begin{defin}\cite{Basarir}
A subset $E\subset \mathbb{N}\times\mathbb{N}$  is said to be uniformly of zero density if and only if the number of elements of $E$ which lie in the rectangle $D$ is $o(pq)$ as $p,q\to\infty$, uniformly in $m,n\geq 0$, where $D=\left\{ (j,k):m\leq j\leq m+p-1, n\leq k\leq n+q-1 \right\}$.
\end{defin}

\begin{lem}\cite{Basarir}\label{lem3.1}
Four-dimensioanl matrix $A=(a_{mnkl})\in([\mathcal{C}_f];\mathcal{C}_{bp})$ with $bp-\lim Ax=[f_2]-\lim x$ if and only if $A$ is regular,that is, $A=(a_{mnkl})\in(\mathcal{C}_{bp};\mathcal{C}_{bp})$ with $bp-\lim Ax=bp-\lim_{kl} x_{kl}$ and 
\begin{eqnarray}
&&\label{eq3.11}\lim_{m,n\to\infty}\sum_{k,l\in E}\left| \Delta_{10}a_{mnkl}\right|\to 0,\\
&&\label{eq3.12}\lim_{m,n\to\infty}\sum_{k,l\in E}\left| \Delta_{01}a_{mnkl}\right|\to 0
\end{eqnarray}
for each set $E$ which is uniformly zero density where 
\begin{eqnarray}\label{eq3.2}
\Delta_{10}a_{mnkl}=a_{mnkl}-a_{mn,k+1,l},~\Delta_{01}a_{mnkl}=a_{mnkl}-a_{mn,k,l+1}
\end{eqnarray}
\end{lem}

\begin{lem}\cite{ZMM}\label{lem3.5}
	The following statements hold:
	\item[(a)] A four dimensional matrix $A=(a_{mnkl})$ is almost $\mathcal{C}_{bp}-$conservative, i.e., $A\in(\mathcal{C}_{bp}:\mathcal{C}_{f})$ iff the following conditions hold
	\begin{eqnarray}
	\label{eq31.0}&&\sup_{m,n\in\mathbb{N}}\sum_{k,l}|a_{mnkl}|<\infty\\
	\nonumber\\
	\label{eq35.1a}&&\exists a_{ij}\in\mathbb{C} \ni bp-\lim_{q,q'\to\infty}a(i,j,q,q',m,n)=a_{ij},\
	\nonumber\\
	&&\textit{ uniformly in } m,n\in\mathbb{N}\textit{ for each }i,j\in\mathbb{N}\\
	\nonumber\\
	\label{eq35.2a}&&\exists u\in\mathbb{C} \ni bp-\lim_{q,q'\to\infty}\sum_{i,j}a(i,j,q,q',m,n)=u,\
	\nonumber\\
	&&\textit{ uniformly in } m,n\in\mathbb{N}\\
	\nonumber\\
	\label{eq35.3a}&&\exists a_{ij}\in\mathbb{C} \ni bp-\lim_{q,q'\to\infty}\sum_{i}|a(i,j,q,q',m,n)-a_{ij}|=0,\
	\nonumber\\
	&&\textit{ uniformly in } m,n\in\mathbb{N}\textit{ for each }j\in\mathbb{N}\\
	\nonumber\\
	\label{eq35.4a}&&\exists a_{ij}\in\mathbb{C} \ni bp-\lim_{q,q'\to\infty}\sum_{j}|a(i,j,q,q',m,n)-a_{ij}|=0,\
	\nonumber\\
	&&\textit{ uniformly in } m,n\in\mathbb{N}\textit{ for each }i\in\mathbb{N}
	\end{eqnarray}
	where $a(i,j,q,q',m,n)=\sum_{k=m}^{m+q}\sum_{l=n}^{n+q'}a_{klij}/[(q+1)(q'+1)]$. In this case, $a=(a_{ij})\in\mathcal{L}_u$ and
	\begin{eqnarray*}
		f_2-\lim Ax=\sum_{i,j}a_{ij}x_{ij}+\left(u-\sum_{i,j}a_{ij}\right)bp-\lim_{i,j\to\infty}x_{ij},
	\end{eqnarray*}
	that is,
	\begin{eqnarray*}
		&&bp-\lim_{q,q'\to\infty}\sum_{i,j}a(i,j,q,q',m,n)x_{ij}=\sum_{i,j}a_{ij}x_{ij}+\left(u-\sum_{i,j}a_{ij}\right)bp-\lim_{i,j\to\infty}x_{ij},\\
		&&\textit{ uniformly in }m,n\in\mathbb{N}.
	\end{eqnarray*}
	\item[(b)] A four dimensional matrix $A=(a_{mnkl})$ is almost $\mathcal{C}_{bp}-$regular, i.e., $A\in(\mathcal{C}_{bp}:\mathcal{C}_{f})_{reg}$ iff
	the conditions (\ref{eq31.0})-(\ref{eq35.4a}) hold with $a_{ij}=0$ for all $i,j\in\mathbb{N}$ and $u=1$
\end{lem}

Now let define the following sets.

\begin{eqnarray*}
	d_1&=&\left\{a=(a_{kl})\in\Omega:\sup_{m,n\in\mathbb{N}}\sum_{k,l}\left|\sum_{j,i=k,l}^{m,n}\left(\frac{-s}{r}\right)^{j-k}\left(\frac{-u}{t}\right)^{i-l}\frac{a_{ji}}{rt}\right|<\infty\right\},\\
	d_2&=&\biggl\{a=(a_{kl})\in\Omega:\exists\beta_{kl}\in\mathbb{C}\ni, \vartheta-\lim_{m,n\to\infty}\sum_{j,i=k,l}^{m,n}\left(\frac{-s}{r}\right)^{j-k}\left(\frac{-u}{t}\right)^{i-l}a_{ji}=\beta_{kl}\biggr\},
\end{eqnarray*}

\begin{eqnarray*}
	d_3&=&\biggl\{a=(a_{kl})\in\Omega:\exists u\in\mathbb{C}\ni, \vartheta-\lim_{m,n\to\infty}\sum_{k,l}\sum_{j,i=k,l}^{m,n}\left(\frac{-s}{r}\right)^{j-k}\left(\frac{-u}{t}\right)^{i-l}\frac{a_{ji}}{rt}=u\biggr\},\\
	d_4&=&\biggl\{a=(a_{kl})\in\Omega:\exists l_0\in\mathbb{N}\ni,\vartheta-\lim_{m,n\to\infty}\sum_{k}\left|\sum_{j,i=k,l_0}^{m,n}\left(\frac{-s}{r}\right)^{j-k}\left(\frac{-u}{t}\right)^{i-l_0}a_{ji}-\beta_{k,l_0}\right|=0\textit{ for all }k\in\mathbb{N}\biggr\},\\
	d_5&=&\biggl\{a=(a_{kl})\in\Omega:\exists k_0\in\mathbb{N}\ni,\vartheta-\lim_{m,n\to\infty}\sum_{l}\left|\sum_{j,i=k_0,l}^{m,n}\left(\frac{-s}{r}\right)^{j-k_0}\left(\frac{-u}{t}\right)^{i-l}a_{ji}-\beta_{k_0,l}\right|=0\textit{ for all }l\in\mathbb{N}\biggr\},\\
	d_6&=&\biggl\{a=(a_{kl})\in\Omega:\vartheta-\lim_{m,n\to\infty}\sum_{k\in E}\sum_{l\in E}\left|\Delta_{01}\left\{\sum_{j,i=k,l}^{m,n}\left(\frac{-s}{r}\right)^{j-k}\left(\frac{-u}{t}\right)^{i-l}\frac{a_{ji}}{rt}\right\}\right|=0\biggr\},\\
	d_7&=&\biggl\{a=(a_{kl})\in\Omega:\vartheta-\lim_{m,n\to\infty}\sum_{k\in E}\sum_{l\in E}\left|\Delta_{10}\left\{\sum_{j,i=k,l}^{m,n}\left(\frac{-s}{r}\right)^{j-k}\left(\frac{-u}{t}\right)^{i-l}\frac{a_{ji}}{rt}\right\}\right|=0\biggr\}.
\end{eqnarray*}

\begin{thm}\label{thm3.1}
The $\beta(\vartheta)-$dual of the space $B[\mathcal{C}_f]$ is the set $\bigcap_{i=1}^7d_i$
\end{thm}

\begin{proof}
Suppose that $a=(a_{mn})\in\Omega$ and $x=(x_{mn})\in B(\mathcal{C}_f)$. We need to show that $\left( \sum_{k,l}^{m,n}a_{kl}x_{kl}\right)_{m,n\in\mathbb{N}}\in CS_{bp}$ for these sequences $a=(a_{mn})\in\Omega$ and $x=(x_{mn})\in B[\mathcal{C}_f]$, that is, $y=Bx\in[\mathcal{C}_f]$ where $y$ is the $B-$transform of $x$ (see the equality (\ref{eq1.3}) and (\ref{eq1.4}) ).  $(m,n)$-th partial sum of $\sum_{k,l}a_{kl}x_{kl}$ is the equality $(3.16)$ which was defined by Tu\v{g} \cite[Theorem 3.11, p.14]{OT2} and the matrix $D=(d_{mnkl})$ which was defined as
\begin{eqnarray}\label{eqy323}
d_{mnkl}=\left\{\begin{array}{ccl}
	\sum_{j,i=k,l}^{m,n}\left(\frac{-s}{r}\right)^{j-k}\left(\frac{-u}{t}\right)^{i-l}\frac{a_{ji}}{rt}&, & 0\leq k\leq m, 0\leq l\leq n; \\
	0&, & \textrm{otherwise}
	\end{array}\right.
\end{eqnarray}
for all $k,l,m,n\in\mathbb{N}$. Because of the hypothesis, one can obtain that $ax\in \mathcal{CS}_{bp}$ whenever $x=(x_{mn})\in B[\mathcal{C}_f]$ if and only if $Dy\in \mathcal{C}_{bp}$ whenever $y=(y_{mn})\in[\mathcal{C}_f]$. Thus, we can equally say that $a=(a_{mn})\in\left\{B[\mathcal{C}_f]\right\}^{\beta(\vartheta)} $ if and only if $D\in([\mathcal{C}_f]: \mathcal{C}_{bp})$. Therefore, the conditions of Lemma \ref{lem3.1} holds with $d_{mnkl}$ instead of $a_{mnkl}$, i.e.,
	\begin{eqnarray*}
		&&\sup_{m,n\in\mathbb{N}}\sum_{k,l}\left|\sum_{j,i=k,l}^{m,n}\left(\frac{-s}{r}\right)^{j-k}\left(\frac{-u}{t}\right)^{i-l}\frac{a_{ji}}{rt}\right|<\infty,\\
		&&\exists\beta_{kl}\in\mathbb{C}\ni, bp-\lim_{m,n\to\infty}\sum_{j,i=k,l}^{m,n}\left(\frac{-s}{r}\right)^{j-k}\left(\frac{-u}{t}\right)^{i-l}a_{ji}=\beta_{kl},\\
		&&\exists u\in\mathbb{C}\ni, bp-\lim_{m,n\to\infty}\sum_{k,l}\sum_{j,i=k,l}^{m,n}\left(\frac{-s}{r}\right)^{j-k}\left(\frac{-u}{t}\right)^{i-l}\frac{a_{ji}}{rt}=u,\\
		&&\exists l_0\in\mathbb{N}\ni, bp-\lim_{m,n\to\infty}\sum_{k}\left|\sum_{j,i=k,l_0}^{m,n}\left(\frac{-s}{r}\right)^{j-k}\left(\frac{-u}{t}\right)^{i-l_0}a_{ji}-\beta_{k,l_0}\right|=0,\\
		&&\textit{ for all }k\in\mathbb{N},\\
		&&\exists k_0\in\mathbb{N}\ni, bp-\lim_{m,n\to\infty}\sum_{l}\left|\sum_{j,i=k_0,l}^{m,n}\left(\frac{-s}{r}\right)^{j-k_0}\left(\frac{-u}{t}\right)^{i-l}a_{ji}-\beta_{k_0,l}\right|=0,\\
		&&\textit{ for all }l\in\mathbb{N},\\
		&&bp-\lim_{m,n\to\infty}\sum_{k\in E}\sum_{l\in E}\left|\Delta_{01}\left\{\sum_{j,i=k,l}^{m,n}\left(\frac{-s}{r}\right)^{j-k}\left(\frac{-u}{t}\right)^{i-l}\frac{a_{ji}}{rt}\right\}\right|=0,\\
		&&bp-\lim_{m,n\to\infty}\sum_{k\in E}\sum_{l\in E}\left|\Delta_{10}\left\{\sum_{j,i=k,l}^{m,n}\left(\frac{-s}{r}\right)^{j-k}\left(\frac{-u}{t}\right)^{i-l}\frac{a_{ji}}{rt}\right\}\right|=0.
\end{eqnarray*}
which is the set $\bigcap_{i=1}^{7}d_i$ as we assumed. 
\end{proof}

\begin{lem}\label{lemm3.1}\cite[Theorem 4.10, p.14]{Orhan}
A four dimensional matrix $A=(a_{mnkl})\in(\mathcal{C}_f:\mathcal{M}_u)$  if and only if
\begin{eqnarray}
\label{eq2.3}&&A_{mn}\in\mathcal{C}_f^{\beta(\vartheta)} ~for ~all~ m,n\in\mathbb{N},\\
\label{eq2.4}&&\sup_{m,n\in\mathbb{N}}\sum_{k,l}|a_{mnkl}|<\infty.
\end{eqnarray}
\end{lem}

The following corollary is the direct consequence of the above Lemma \ref{lemm3.1}, since $([\mathcal{C}_f]:\mathcal{M}_u)\subset(\mathcal{C}_f:\mathcal{M}_u)$ and since $\left\{[\mathcal{C}_f]\right\}^{\beta(\vartheta)}\subset\left\{\mathcal{C}_f\right\}^{\beta(\vartheta)}$ holds.
\begin{cor}\label{cor3.1}
A four dimensional matrix $A=(a_{mnkl})\in([\mathcal{C}_f]:\mathcal{M}_u)$  if and only if the $A_{mn}\in\left\{[\mathcal{C}_f]\right\}^{\beta(\vartheta)}$ for all $m,n\in\mathbb{N}$ and (\ref{eq2.4}) hold.
\end{cor}

\begin{thm}
The $\gamma-$dual of the space $\left\{B[\mathcal{C}_f]\right\}^{\gamma}=d_1\cap CS_{\vartheta}$.
\end{thm}
\begin{proof}
To prove this theorem we need to show that $\left( \sum_{k,l}^{m,n}a_{kl}x_{kl}\right)_{m,n\in\mathbb{N}}\in BS$ by supposing $a=(a_{mn})\in\Omega$ and $x=(x_{mn})\in B[\mathcal{C}_f]$ where $y=Bx\in[\mathcal{C}_f]$. If we follow the similar way with the Theorem \ref{thm3.1}, we can say that $ax\in\mathcal{BS}$ whenever $x=(x_{mn})\in B[\mathcal{C}_f]$ if and only if $Dy\in\mathcal{M}_u$ whenever $y=(y_{mn})\in[\mathcal{C}_f]$, where the matrix $D=(d_{mnkl})$ which was defined in the Theorem \ref{thm3.1} as (\ref{eqy323}). Consequently, we can say that $a=(a_{mn})\in\left\{B[\mathcal{C}_f]\right\}^{\gamma}$ if and only if $D\in([\mathcal{C}_f]:\mathcal{M}_u)$. Hence, the conditions of Corollary \ref{cor3.1} holds with the matrix $D=(d_{mnkl})$ instead of the matrix $A=(a_{mnkl})$. Therefore, the The $\gamma-$dual of the space $\left\{B[\mathcal{C}_f]\right\}^{\gamma}$ is the set $d_1\cup CS_{\vartheta}$ which completes the proof.
\end{proof}

\section{Matrix Transformations related to the Sequence Space $B[\mathcal{C}_{f}]$}

Characterization of four-dimensional matrices has an importance in four-dimensional matrix transformations. Some significant classes have been characterized by several mathematicians (see \cite{ZM,Mursaleen,MYFB22,Basarir,ZMM,MYFB}). In this present section, to fill a gap in the concerned literature, we characterize some new four-dimensional matrix classes $([\mathcal{C}_f];\mathcal{C}_{f})$, $(B[\mathcal{C}_f];\mathcal{C}_{f})$ and $(B[\mathcal{C}_{f}]:\mathcal{M}_{u})$ after stating some needed Lemmas. Then, we complete this section with some significant results.

\begin{lem}\cite{MES}\label{lem3.3}
	A four dimensional matrix $A=(a_{mnkl})$ is almost regular, i.e., $A\in(\mathcal{C}_{bp}:\mathcal{C}_{f})_{reg}$ iff the condition (\ref{eq31.0}) and the following conditions hold
	\begin{eqnarray}
	\nonumber\\
	\label{eq33.1}&&\lim_{q,q'\to\infty}a(i,j,q,q',m,n)=0,\
	\nonumber\\
	&&\textit{ uniformly in } m,n\in\mathbb{N}\textit{ for each }i,j\in\mathbb{N},\\
	\nonumber\\
	\label{eq33.2}&&\lim_{q,q'\to\infty}\sum_{i,j}a(i,j,q,q',m,n)=1,\
	\nonumber\\
	&&\textit{ uniformly in } m,n\in\mathbb{N},\\
	\nonumber\\
	\label{eq33.3}&&\lim_{q,q'\to\infty}\sum_{i}\left|a(i,j,q,q',m,n)\right|=0,\
	\nonumber\\
	&&\textit{ uniformly in } m,n\in\mathbb{N}\textit{ for each }j\in\mathbb{N},\\
	\nonumber\\
	\label{eq33.4}&&\lim_{q,q'\to\infty}\sum_{j}\left|a(i,j,q,q',m,n)\right|=0,\
	\nonumber\\
	&&\textit{ uniformly in } m,n\in\mathbb{N}\textit{ for each }i\in\mathbb{N},
	\end{eqnarray}
\end{lem}
where $a(i,j,q,q',m,n)$ is defined as in Lemma \ref{lem3.5}.
\begin{lem}\cite{MM}\label{lem3.4}
	A four dimensional matrix $A=(a_{mnkl})$ is almost strongly regular, i.e., $A\in(\mathcal{C}_{f}:\mathcal{C}_{f})_{reg}$ iff $A$ is almost regular and the following two conditions hold
	\begin{eqnarray}
	\label{eq34.1}&&\lim_{q,q'\to\infty}\sum_{i}\sum_{j}\left|\Delta_{10}a(i,j,q,q',m,n)\right|=0\textit{ uniformly in } m,n\in\mathbb{N},\\
	\label{eq34.2}&&\lim_{q,q'\to\infty}\sum_{j}\sum_{i}\left|\Delta_{01}a(i,j,q,q',m,n)\right|=0\textit{ uniformly in } m,n\in\mathbb{N},
	\end{eqnarray}
	where
	\begin{eqnarray*}
		&&\Delta_{10}a(i,j,q,q',m,n)=a(i,j,q,q',m,n)-a(i+1,j,q,q',m,n),\\
		&&\Delta_{01}a(i,j,q,q',m,n)=a(i,j,q,q',m,n)-a(i,j+1,q,q',m,n).
	\end{eqnarray*}
\end{lem}

\begin{thm}\label{thm4.1}
Four-dimensioanl matrix $A=(a_{mnkl})\in([\mathcal{C}_f];\mathcal{C}_{f})$ with $f_2-\lim Ax=[f_2]-\lim_{kl} x_{kl}$ if and only if $A$ is almost $\mathcal{C}_{bp}-$regular,i.e., $A=(a_{mnkl})\in(\mathcal{C}_{bp};\mathcal{C}_{f})$ with $f_2-\lim Ax=bp-\lim_{kl} x_{kl}$ and 
\begin{eqnarray}\label{eq4.1}
	\sum_{k,l\in E}\left| \Delta_{11}a_{mnkl}\right|\to 0, ~~as~ m,n\to\infty
\end{eqnarray}
for each set $E$ which is uniformly zero density where 
\begin{eqnarray*}
\Delta_{11}a_{mnkl}=a_{mnkl}-a_{mn,k+1,l}-a_{mn,k,l+1}+a_{mn,k+1,l+1}
\end{eqnarray*}

\end{thm}

\begin{proof}
$\Rightarrow:$ Suppose that $A=(a_{mnkl})\in([\mathcal{C}_f];\mathcal{C}_{f})$ with $f_2-\lim Ax=[f_2]-\lim_{kl} x_{kl}$. Then $Ax=y$ exists and is in $\mathcal{C}_{f}$ for all sequences $x=x_{kl}\in[\mathcal{C}_{f}]$. Since the inclusion $[\mathcal{C}_f]\subset\mathcal{C}_{f}$ strictly hold and each inclusion is proper, then one can obtain that $x=(x_{kl})$ is also almost convergent to zero and is also in $\mathcal{C}_{f}$ Thus, The matrix $A=(a_{mnkl})$ is almost regular,i.e., $A=(a_{mnkl})\in(\mathcal{C}_f;\mathcal{C}_{f})$ with $f_2-\lim Ax=f_2-\lim_{kl} x_{kl}$. By Lemma \ref{lem3.4} we can also say that $A=(a_{mnkl})\in(\mathcal{C}_{bp};\mathcal{C}_{f})$ with $f_2-\lim Ax=bp-\lim_{kl} x_{kl}$ which prove that the matrix $A=(a_{mnkl})$ is strongly almost $\mathcal{C}_{bp}-$regular with $f_2-\lim Ax=bp-\lim_{kl} x_{kl}$. Therefore, the conditions of the Lemma \ref{lem3.5}(b) and the conditions (\ref{eq34.1})-(\ref{eq34.2}) of the Lemma \ref{lem3.4} satisfied with $A=(a_{MNkl})$. Thus,
\begin{eqnarray*}
p-\lim_{q,q'\to\infty}\frac{1}{(q+1)(q'+1)}\sum_{k=m}^{m+q}\sum_{l=n}^{n+q'}(Ax)_{kl}=y_{MN}, ~~and~~f_2-\lim_{M,N\rightarrow\infty} y_{MN}=0,
\end{eqnarray*}
and since $\mathcal{C}_{f}\subset\mathcal{M}_{u}$ holds, then $y=(y_{MN})$ is also in$\mathcal{M}_{u}$, i.e., there exists a positive real number $K$ such that $\|y\|_{\infty}=\sup_{M,N\in\mathbb{N}}|y_{MN}|\leq K<\infty$. Moreover, since $x=(x_{kl})$ is strongly almost convergent to zero, then 

\begin{eqnarray*}
	\sum_{k=m}^{m+q}\sum_{l=n}^{n+q'}|x_{kl}|<\epsilon(q+1)(q'+1)
\end{eqnarray*}
holds for every $\epsilon>0$, $q,q'\geq 1$ and uniformly in $m,n\in\mathbb{N}$.

Now, let consider the set $E=\left\{ (k,l):|x_{kl}|\geq \frac{\epsilon}{(q+1)(q'+1)} \right\}$. Then the number of element of the set $E$ which lie in the rectangle $D=\{ (k,l):m\leq k<m+q, n\leq l<n+q' \}$ is $o((q+1)(q'+1))$ as $q,q'\to\infty$, uniformly in $m,n\in\mathbb{N}$. Therefore, the set $E$ is uniformly of zero density.

Since $A=(a_{mnkl})$ is almost $\mathcal{C}_{bp}-$regular matrix, then the condition (\ref{eq31.0}) holds for $A$, that is,
\begin{eqnarray*}
\|A\|=\sup_{M,N\in\mathbb{N}}\sum_{k,l}|a_{MNkl}|<\infty. 
\end{eqnarray*}
After having all above preparations, we have 
\begin{eqnarray*}
\left| \frac{1}{(q+1)(q'+1)}\sum_{m=0}^{\infty}\sum_{n=0}^{\infty}a_{MNmn}\sum_{k=m}^{m+q}\sum_{l=n}^{n+q'}x_{kl} \right|&\leq&\frac{1}{(q+1)(q'+1)}\sum_{m=0}^{\infty}\sum_{n=0}^{\infty}|a_{MNmn}|\sum_{k=m}^{m+q}\sum_{l=n}^{n+q'}|x_{kl}|\\
&\leq&\epsilon \|A\|.
\end{eqnarray*}

Moricz and Rhodes \cite[Theorem 1]{Moricz} formulate the following sum after several calculations (see the formulas (2), (3) and (7)). Here we have the same facts that we write,

\begin{eqnarray*}
\frac{1}{(q+1)(q'+1)}\sum_{m=0}^{\infty}\sum_{n=0}^{\infty}a_{MNmn}\sum_{k=m}^{m+q}\sum_{l=n}^{n+q'}x_{kl}&=&o(0)+y_{MN}\\
\nonumber&+&\sum_{k=q}^{\infty}\sum_{l=q'}^{\infty}x_{kl}\left\{\frac{1}{(q+1)(q'+1)}\sum_{m=k-q}^{k}\sum_{n=l-q'}^{l}(a_{MNmn}-a_{MNkl})\right\}
\end{eqnarray*}
Our aim here is to show the right hand side of the above equation is as small as we wish as $M,N\to\infty$. To accomplish this, let's consider $0\leq \pi\leq q$ and $0\leq \rho\leq q'$, then
\begin{eqnarray*}
&&\left|\sum_{k=q}^{\infty}\sum_{l=q'}^{\infty}x_{kl}\left\{\frac{1}{(q+1)(q'+1)}\sum_{m=k-q}^{k}\sum_{n=l-q'}^{l}(a_{MNmn}-a_{MNkl})\right\}\right|\\
&\leq&\frac{ \|x\|}{(q+1)(q'+1)}\sum_{k=q}^{\infty}\sum_{l=q'}^{\infty}\left|\sum_{m=k-q}^{k}\sum_{n=l-q'}^{l}(a_{MNmn}-a_{MNkl}) \right|\\
&\leq&\frac{ \|x\|}{(q+1)(q'+1)}\sum_{k=q}^{\infty}\sum_{l=q'}^{\infty}\sum_{m=k-q}^{k}\sum_{n=l-q'}^{l}\left|a_{MNmn}-a_{MNkl} \right|\\
&=&\frac{ \|x\|}{(q+1)(q'+1)}\sum_{k=q}^{\infty}\sum_{l=q'}^{\infty}\sum_{\pi=0}^{q}\sum_{\rho=o}^{q'}\left|a_{MN,\pi+k-q,\rho+l-q'}-a_{MNkl} \right|\\
&=&\frac{ \|x\|}{(q+1)(q'+1)}\sum_{\pi=0}^{q}\sum_{\rho=o}^{q'}\sum_{k=q}^{\infty}\sum_{l=q'}^{\infty}\left|a_{MN,\pi+k-q,\rho+l-q'}-a_{MNkl} \right|\\
&\leq&\frac{ \|x\|}{(q+1)(q'+1)}\sum_{\pi=0}^{q}\sum_{\rho=o}^{q'}\left\{(q-\pi)\sum_{k=0}^{\infty}\sum_{l=0}^{\infty}\left|\Delta_{10}a_{MNkl}\right|+(q'-\rho)\sum_{k=0}^{\infty}\sum_{l=0}^{\infty}\left|\Delta_{01}a_{MNkl}\right|\right\}\\
&\leq&(q-\pi)\left(\frac{ \|x\|}{(q+1)(q'+1)}\sum_{k\in E}\sum_{l\in E}\left|\Delta_{10}a_{MNkl}\right|+\epsilon\sum_{k=0}^{\infty}\sum_{l=0}^{\infty}\left|\Delta_{10}a_{MNkl}\right|\right)\\
&&+(q'-\rho)\left(\frac{ \|x\|}{(q+1)(q'+1)}\sum_{k\in E}\sum_{l\in E}\left|\Delta_{01}a_{MNkl}\right|+\epsilon\sum_{k=0}^{\infty}\sum_{l=0}^{\infty}\left|\Delta_{01}a_{MNkl}\right|\right)\\
&\leq&\frac{ q\|x\|}{(q+1)(q'+1)}\sum_{k\in E}\sum_{l\in E}\left|\Delta_{10}a_{MNkl}\right|+q\epsilon\sum_{k=0}^{\infty}\sum_{l=0}^{\infty}\left|\Delta_{10}a_{MNkl}\right|\\
&&+\frac{q' \|x\|}{(q+1)(q'+1)}\sum_{k\in E}\sum_{l\in E}\left|\Delta_{01}a_{MNkl}\right|+q'\epsilon\sum_{k=0}^{\infty}\sum_{l=0}^{\infty}\left|\Delta_{01}a_{MNkl}\right|.
\end{eqnarray*}
Then the proof of sufficiency follows by letting $M,N\to\infty$.

$\Leftarrow:$  Suppose $A=(a_{mnkl})$ be an almost $\mathcal{C}_{bp}-$regular matrix, that is, $A=(a_{mnkl})\in(\mathcal{C}_{bp};\mathcal{C}_{f})$ with $f_2-\lim Ax=bp-\lim_{kl} x_{kl}$. Then $Ax$ exists and is in $\mathcal{C}_{f}$ for all sequences $x=x_{kl}\in\mathcal{C}_{bp}$. Then the matrix $A=(a_{mnkl})$ satisfy the conditions (\ref{eq31.0})-(\ref{eq35.4a}) of Lemma \ref{lem3.5}. 

Now, let suppose that the condition (\ref{eq4.1}) not satisfied. Let $E$ be any set which is uniformly of zero density and $x=(x_{kl})$ be a strongly almost convergent double sequence. So, $x=(x_{kl})$ is bounded. If we define the following sequence $z=(z_{kl})$ by
\begin{eqnarray}
z_{kl}=\left\{\begin{array}{ccl}
x_{kl} &,&  (k,l)\in E\\
0 &,&  othervise
\end{array}\right.
\end{eqnarray}
and the sequence $y=(y_{kl})$ by 
\begin{eqnarray}
y_{kl}=z_{kl}-z_{k+1,l}-z_{k,l+1}+z_{k+1,l+1},\textit{ for each }k,l>1~and~ y_{11}=z_{11}
\end{eqnarray}
Since $E$ is uniformly of zero density, then clearly $[f_2]-\lim_{k,l\to\infty}y_{kl}=0$ such that 
\begin{eqnarray*}
\sum_{k=0}^{\infty}\sum_{l=0}^{\infty}|a_{MNkl}y_{kl}|\to 0,~~~asM,N\to\infty,
\end{eqnarray*}
But
\begin{eqnarray*}
	\sum_{k=0}^{\infty}\sum_{l=0}^{\infty}|a_{MNkl}y_{kl}|=\sum_{k=0}^{\infty}\sum_{l=0}^{\infty}|\Delta_{11}a_{MNkl}z_{kl}|=\sum_{k\in E}^{\infty}\sum_{l\in E}^{\infty}|\Delta_{11}a_{MNkl}x_{kl}|
\end{eqnarray*}
So, this is a contradiction according to our assumption. So the condition (\ref{eq4.1}) is necessity to prove $y=(y_{kl})\in\mathcal{C}_f$ for all $x=(_{kl})\in[\mathcal{C}_f]$. This fact complete the proof.
\end{proof}

\begin{thm}
	Four-dimensional matrix $A=(a_{mnkl})\in(B[\mathcal{C}_f];\mathcal{C}_{f})$ with $f_2-\lim Ax=[f_2]-\lim_{kl} x_{kl}$ if and only if $A$ is almost $B(\mathcal{C}_{bp})-$regular,i.e., $A=(a_{mnkl})\in(B(\mathcal{C}_{bp});\mathcal{C}_{f})$ with $f_2-\lim Ax=bp-\lim_{kl} x_{kl}$ and 
	\begin{eqnarray}
	\sum_{k,l\in E}\left| \Delta_{11}e_{mnkl}\right|\to 0, ~~as~ m,n\to\infty
	\end{eqnarray}
	for each set $E$ which is uniformly zero density where $\Delta_{11}a_{mnkl}$ defined as (\ref{eq3.2}) and
	\begin{eqnarray*}
	e_{mnkl}=\sum_{i,j=k,l}^{m,n}\left(\frac{-s}{r}\right)^{i-k}\left(\frac{-u}{t}\right)^{j-l}\frac{a_{mnij}}{rt}
	\end{eqnarray*}
\end{thm}

\begin{proof}
Suppose that the matrix $A=(a_{mnkl})\in(B[\mathcal{C}_f]:\mathcal{C}_f)$. Then, $Ax$ exists and is in $\mathcal{C}_f$ for all $x=(x_{mn})\in B[\mathcal{C}_f]$ which implies that $Bx=(Bx)_{mn}\in [\mathcal{C}_f]$. We have the following equality derived from the $(m,n)th-$partial sum of the series $\sum_{k,l}a_{mnkl}x_{kl}$ with respect to the relation between terms of $x=(x_{kl})$ and $y=(y_{kl})$,
\begin{eqnarray}\label{eq4.4} \sum_{k,l=0}^{m,n}a_{mnkl}x_{kl}&=&\sum_{k,l=0}^{m,n}a_{mnkl}\sum_{j,i=0}^{k,l}\left(\frac{-s}{r}\right)^{k-j}\left(\frac{-u}{t}\right)^{l-i}\frac{y_{ji}}{rt}\\
&=&\sum_{k,l=0}^{m,n}\sum_{j,i=k,l}^{m,n}\left(\frac{-s}{r}\right)^{j-k}\left(\frac{-u}{t}\right)^{i-l}\frac{a_{mnji}}{rt}y_{kl} \nonumber\\
&=&(Ey)_{mn},\nonumber
\end{eqnarray}
where the four-dimensional matrix $E=(e_{mnkl})$ is defined as in \cite[p.16]{Orhan} by
\begin{eqnarray*}
	e_{mnkl}=\left\{\begin{array}{ccl}
		\sum_{j,i=k,l}^{m,n}\left(\frac{-s}{r}\right)^{j-k}\left(\frac{-u}{t}\right)^{i-l}\frac{a_{mnji}}{rt}&, & 0\leq k\leq m, 0\leq l\leq n; \\
		0&, & \textrm{otherwise}
	\end{array}\right.
\end{eqnarray*}
for all $m,n\in\mathbb{N}$. Then, by taking $f_2-$limit on (\ref{eq4.4}) as $q,q'\to\infty$, we have $Ax=Ey$. Therefore, $Ey\in\mathcal{C}_f$ whenever $y\in[\mathcal{C}_f]$ , that is, $E\in([\mathcal{C}_f]:\mathcal{C}_f)$. Hence, the conditions of Theorem \ref{thm4.1} hold with $E=(e_{mnkl})$ instead of $A=(a_{mnkl})$. This completes the proof.
\end{proof}

\begin{cor}
A four dimensional matrix $A=(a_{mnkl})\in(B[\mathcal{C}_f]:\mathcal{M}_u)$  if and only if $A_{mn}\in\{B[\mathcal{C}_f]\}^{\beta(\vartheta)}$ and the condition (\ref{eq2.4}) hold with $e_{mnkl}$ instead of $a_{mnkl}$.
\end{cor}

\begin{cor}
	Four-dimensional matrix $A=(a_{mnkl})\in(B[\mathcal{C}_f];\mathcal{C}_{bp})$ with $bp-\lim Ax=[f_2]-\lim x$ if and only if  $A=(a_{mnkl})\in(B(\mathcal{C}_{bp});\mathcal{C}_{bp})$ with $bp-\lim Ax=bp-\lim_{kl} x_{kl}$ and the conditions (\ref{eq3.11})-(\ref{eq3.12}) hold with $e_{mnkl}$ instead of $a_{mnkl}$.
\end{cor}

\begin{lem} \cite[Theorem 4.7]{MYFB22}
	Let $\lambda$ and $\mu$ represent any double sequence space, and the elements of the four dimensional matrices $A=(a_{mnkl})$ and $G=(g_{mnkl})$ are connected with the relation
	\begin{eqnarray}\label{eq4.9}
	g_{mnkl}=\sum_{i,j=0}^{m,n}b_{mnij}(r,s,t,u)a_{ijkl} ~\textrm{ for all }~m,n,k,l\in\mathbb{N}.
	\end{eqnarray}
	Then, $A\in(\mu:B(\lambda))$ if and only if $G\in(\mu:\lambda)$.
\end{lem}

\begin{cor}
	A four dimensional matrix $A=(a_{mnkl})\in([\mathcal{C}_f]:B(\mathcal{M}_u))$  if and only if $A_{mn}\in\{[\mathcal{C}_f]\}^{\beta(\vartheta)}$ and the condition (\ref{eq2.4}) hold with $g_{mnkl}$ instead of $a_{mnkl}$.
\end{cor}

\begin{cor}
	Four-dimensional matrix $A=(a_{mnkl})\in([\mathcal{C}_f];B(\mathcal{C}_{bp}))$ with $bp-\lim Ax=[f_2]-\lim x$ if and only if  $A=(a_{mnkl})\in(\mathcal{C}_{bp};B(\mathcal{C}_{bp}))$ with $bp-\lim Ax=bp-\lim_{kl} x_{kl}$ and the conditions (\ref{eq3.11})-(\ref{eq3.12}) hold with $g_{mnkl}$ instead of $a_{mnkl}$.
\end{cor}

\section{conclusion}

The matrix domain on some almost convergent single sequence spaces have been done by several mathematicians (see \cite{FBMK,OTFB2,KKMS,KKMS2}). The idea of calculating four-dimensional matrix domain on almost convergent double sequences spaces is a new subject and studied by a few mathematicians (see \cite{Orhan,Orhan 2,Husamettin}). 

In this paper, as a natural continuation of the papers \cite{Orhan,Orhan 2,Husamettin} we calculated the four-dimensional generalized difference matrix $B(r,s,t,u)$ domain $B[\mathcal{C}_f]$ and $B[\mathcal{C}_{f_0}]$ on the spaces of strongly almost convergent and strongly almost null double sequences $[\mathcal{C}_f]$ and $[\mathcal{C}_{f_0}]$, respectively. We proved some strict inclusion relations and calculated the dual spaces of space $B[\mathcal{C}_f]$. Characterization of the four-dimensional matrix class $([\mathcal{C}_f];\mathcal{C}_{f})$ was an open problem and we put the necessary and sufficient conditions of four-dimensional matrix mapping on the class $([\mathcal{C}_f];\mathcal{C}_{f})$ and proved it. Then we also characterized some other matrix classes which set up with $B(r,s,t,u)-$domain of the space $[\mathcal{C}_f]$. 

Similar works can be done by letting other four-dimensional matrices domain on almost convergent double sequences and new matrix classes including $[\mathcal{C}_f]$ can be characterized for the further studies.

\section*{Competing interests}
The author OT declares here that there is no competing interests.

\section*{Author's contributions}
The author Orhan Tu\v{g} defined new strongly almost convergent and strongly almost null double sequence spaces which were derived as the domain of four-dimensional generalized difference matrix $B(r,s,t,u)=(b_{mnkl}(r,s,t,u))$ and proved some topological results. Moreover, Orhan Tu\v{g} computed the $\alpha-$, $\beta(\vartheta)-$ and $\gamma-$duals of this new strongly almost convergent double sequence spaces and characterized some new matrix classes. In the last section, some results were stated and some open problems were given by the author. The author read and approved the final manuscript\ldots

\section*{Acknowledgement}
I would like to thank Prof. Dr. Viladimir Rakocevic and Prof. Dr. Eberhard Malkowsky for their unlimited supports and valuable suggestions on all my academic works. I also would like thank to the colleagues in the Faculty of Education, Tishk International University for their valuable comments on the latest version of this paper. This work was supported by the Research Center of Tishk International University, Erbil-IRAQ.
\section*{Article Information}
The previous related works of this paper were published in "Journal of Inequalities and Applications".

\end{document}